\documentclass[11pt]{amsart}
\usepackage{amsmath,amsfonts,amssymb,latexsym}
\usepackage{graphicx}
\usepackage[T1]{fontenc}
\usepackage{times, mathptm}
\newtheorem{theorem}{\bf \mbox{Theorem}}[section]
\newtheorem{proposition}[theorem]{\bf \mbox{Proposition}}
\newtheorem{lemma}[theorem]{\bf \mbox{Lemma}}
\newtheorem{definition}[theorem]{\bf \mbox{Definition}}

\newtheorem{remark}{\bf Remark}

\newcommand{\field}[1]{\mathbb{#1}}

\newcommand{\bu}{{\bf u}}

\newcommand{\N}{\field{N}}
\newcommand{\R}{\field{R}}

\def\N{\mathbb{N}}

\begin{document}
\title{Thin-Thick decomposition for real definable isolated singularities}
\author{Lev Birbrair, Alexandre Fernandes, Vincent Grandjean}
\address{ Departamento de Matem\'atica, UFC,
Av. Humberto Monte s/n, Campus do Pici Bloco 914,
CEP 60.455-760, Fortaleza-CE, Brasil}
\email{birb@ufc.br}
\email{alexandre.fernandes@ufc.br}
\email{vgrandjean@mat.ufc.br}
\thanks{We are very pleased to thank Alberto Verjosky and Andrei Gabrielov for useful comments and suggestions.}
\subjclass[2010]{Primary ,Secondary }

\date{\today}
\maketitle
%
%
%
\begin{abstract}  
Two subset germs of Euclidean spaces are called blow-spherically equivalent, if their spherical modifications 
are homeomorphic and the 
homeomorphism induces homeomorphic tangent links.  
Blow-spherical equivalence is stronger than the topological equivalence but weaker than the Lipschitz equivalence.  
We introduce the thin-thick decomposition of an isolated singularity germ - which happens to be a natural blow-spherical
invariant. 
This decomposition is a generalization of the thin-thick decomposition of normal complex
surface singularity germs introduced in \cite{BNP}.
\end{abstract}
\section{Introduction}
Recent investigations in Lipschitz Geometry of complex algebraic surface singularities are closely related to the notion
of thin-thick decomposition discovered in \cite{BNP}. Indeed a germ of a normal complex algebraic surface 
can be divided into two zones. The first one, called the thick zone, contains "almost everything" and  
is "almost metrically conical". The other zone, called the thin zone, is not metrically conical, has density zero at the origin  
and moreover contains some of the most important invariants of the Lipschitz geometry of
singular subset germs such as fast loops, choking-horns or separating sets (\cite{BFN1,BFN2,BFGoS}). 
The topology of the thin-thick decomposition is Lipschitz invariant and moreover the complete inner Lipschitz 
invariant of normal complex algebraic surface singularity germs can be obtained as the iterated thin-thick decomposition \cite{BNP}. 

The present paper is devoted to a general vision of the notion of thin-thick decomposition for rather general sets with 
isolated singularities. Instead of complex algebraic sets we consider definable sets with isolated singularities. 
A spherical modification of a germ at a singular point is obtained as follows. Suppose for simplicity that the singular 
point is the origin. Let us intersect our set with a sphere of radius $r$ and rescale this intersection by the factor 
$\frac{1+r}{r}$. The closure of the union (over all positive $r$) of these sets is called the spherical modification. The intersection of this 
closure with the unit sphere is called the tangent link. The tangent cone of the set at the origin is the 
(non-negative) cone over the tangent link. 

The paper distinguishes two cases: the normally embedded case and the non-normally embedded case (the general case). 
Let us start with the normally embedded case. A point on the tangent link is called simple if the spherical modification 
nearby this point is a topological manifold with boundary. The set of simple points is open and 
of codimension one in the subset once nonempty. 
It is the thick part of the tangent link. Its complement has dimension strictly smaller than the dimension 
of the set of the simple points if there are simple points. This complement is called the thin part of the 
tangent link. Thus we obtained the thin-thick decomposition of the tangent link. This decomposition of the tangent link 
is very closely related to the think-thick decomposition of the subset itself. 
The intersection of the subset with a sufficiently thin (but not too thin) horn-like neighborhood of the cone over the thin part 
of the tangent cone  is the thin zone of the subset germ. The thick zone of the germ is the complement of the thin zone. 
The thick zone of the set is homeomorphic to a cone over the thick part of the tangent cone.

We introduce in this paper a new equivalence relation for isolated singularity subset germs, called
blow-spherical equivalence, in order to investigate local inner metric properties of the set germ (such as 
for instance the Thick-Thin decomposition). The  blow-spherical equivalence 
consists of a homeomorphism between the respective spherical modifications and also mapping homeomorphically 
the tangent link to the tangent link. 
Occurrences of such blow-spherical homeomorphisms are for instance Bi-Lipschitz maps and also homeomorphisms 
differentiable at the singular points \cite{GL}. 
We prove here that the topology of the thin-thick decomposition is a blow-spherical invariant (Theorem 
\ref{thm:main-NE-XY} and Theorem \ref{thm:main-nNE-XY}). 
Birbrair, Neumann and Pichon, while working on Lipschitz Geometry of complex singularities, 
came up with the notion of locally conical sets. An isolated singular point of a closed semialgebraic subset is called 
locally metrically conical, if for any direction of the tangent cone one can find a semialgebraic bi-Lipschitz embedding of a small straight 
cone to the set, such that the image of the center line of the cone will be tangent to this direction. A natural 
question to ask is whether a locally metrically conical singularity is metrically conical, i.e. bi-Lipschitz equivalent to a cone.  
This question remains open. In the paper we show that blow-spherical geometry is simpler than the Lipschitz geometry 
in the following sense: If all tangent directions at the singular point of an isolated singular definable subset germ are simple, 
then the singularity germ is blow-spherically equivalent to a cone. 
In particular, we prove that locally metrically conical singularities are 
blow-spherically conical.

The other part of the paper is devoted to the non-normally embedded case. 
In this case, any point of the tangent link point embedded in the spherical modification of the subset germ may not 
be a boundary point. 
The simple points of the tangent link of such a non-normally embedded germ are the points at which the spherical 
modification has a so called open pamphlet structure: a finite collection of topological manifolds with boundary, glued along the same 
boundary. The number of the corresponding manifolds is called the multiplicity of the simple point. 
The thin-thick decomposition of the tangent link has a similar structure as that of the normally embedded case. 
The thin-thick decomposition of the subset germ itself is also obtained in a similar way, but the multiplicity plays an 
important part. The topology of the components of the thick zone are related to the topology of the corresponding 
parts of the metric tangent cone (see \cite{BL} or \cite{BFN2}). Finally we show that the topology of the thin-thick 
decomposition of the set, equipped with the multiplicity, is an invariant for the blow-spherical equivalence. We observe that 
the thin-thick decomposition from \cite{BNP} is a particular case of our construction.

\section{Normally Embedded Case}

\medskip
Let $X\subset\R^n$ be a closed definable set in a given polynomially bounded o-minimal structure expanding the field 
of real numbers (\cite{vdDMi,vdD,Co}). Assume that the origin $0$ is a singular point of $X$.
 
Let $S(0,\epsilon)$ be the Euclidean sphere of $\R^n$ of radius $\epsilon$ centered at $0$.

Let $\tilde{X}\subset\R^n$ be the topological closure of the set:
$$\bigcup_{0<\epsilon <\epsilon_{0}}\frac{1+\epsilon}{\epsilon}\cdot (X\cap S(0,\epsilon)),$$ 
where $\epsilon_0$ is a value such that the family of sets $\{X\cap S(0,\epsilon)\}_{0<\epsilon <\epsilon_0}$ 
is topologically trivial. The set $\tilde{X}$ is called the \emph{spherical modification} of $X$ at $0$ and is also 
a definable subset.

\medskip
Let $L_0X$ be the set $\tilde{X}\cap S(0,1)$. Then, the tangent cone $T_0X$ of $X$ at $0$ is the straight cone over $L_0X$,
namely $T_0 X = \cup_{u\in L_0 X}(\R_{\geq 0} \cdot u)$.

\begin{definition}{\rm The germs $(X,0)$ and $(Y,0)$ are called \emph{blow-spherical equivalent} if there exists a homeomorphism $h:\tilde{X}\rightarrow\tilde{Y}$ sending $L_0X$ to $L_0Y$.}
\end{definition}

\begin{remark}{\rm  Let $(X,0)$ and $(Y,0)$ be bi-Lipschitz equivalent, such that the bi-Lipschitz map $(X,0)\rightarrow (Y,0)$ 
is definable, then $(X,0)$ and $(Y,0)$ are blow-spherical equivalent.}
\end{remark}

The map $\pi : \tilde{X}\rightarrow X$ defined by 
$$\pi (x)=\frac{\|x\|-1}{\|x\|}x$$ defines a homeomorphism between $\tilde{X}\setminus L_0X$ 
and $X\cap B(0,\epsilon_0)\setminus 0$ and $\pi(L_0X)=0$.
We recall that a subset of $\R^n$ is normally embedded if the outer metric and the inner metric are 
bi-Lipschitz equivalent \cite{BM}. 
\begin{definition} {\rm Let $X$ be a normally embedded definable set. A point $v\in L_0X$ is called a \emph{simple direction} 
of $X$ at $0$ if $\tilde{X}$ is a topological submanifold with boundary near $v$.}
\end{definition}

\begin{remark}{\rm Let $X$ be a normally embedded definable set. The set of simple directions is definable and open in $L_0X$.}
\end{remark}

\begin{proposition} Let $X$ be normally embedded definable set. Then $X$ has simple directions at $0$ if, 
and only if, ${\rm dim}T_0X$=${\rm dim}_0X$.
\end{proposition}

\begin{proof} It is clear that if ${\rm dim}_0X<{\rm dim T_0X}$, then there is no simple directions. Now, let us 
suppose that ${\rm dim}_0X={\rm dim T_0X}$. Consider a triangulation of $\tilde{X}$ compatible with $L(X,0)$. Then, 
one has a point $v$ belonging to a ${\rm dim}_0X-1$ face of the triangulation of $L(X,0)$. It is a simple direction. 
Notice that $v$ cannot belong to more than one face like that because $X$ is normally embedded. 
\end{proof}

\begin{definition}{\rm  A singular point $0\in X$ is called \emph{blow-spherical conical} if the germ of $X$ at $0$ is 
bow-spherically equivalent to the germ of the straight cone over $X\cap S^{d-1}(0,\epsilon)$ for sufficiently small 
$\epsilon  >0$, where $d$ is the dimension of the germ $X$ at $0$ and $S^{d-1}(0,\epsilon)$ is the Euclidean sphere 
of $\R^d$ of radius $\epsilon$ centered at the origin.}
\end{definition}

\begin{theorem}\label{conical-structure} A definable and normally embedded set germ $(X,0)$, with 
an isolated singularity at $0$ is blow-spherical conical if, and only if, its tangent link $L_0X$ consists 
only of simple directions.
\end{theorem}

As a first step towards this goal, we recall Brown Theorem about the existence of collar neighborhoods in metric spaces.

\bigskip
A subset $B$ of a topological space $X$ is \em collared \em  if there exists a
homeomorphism $h:B\times [0,1[ \to \mathcal{U}$ onto a neighborhood $\mathcal{U}$ of $B$ such that
$h(b,0) = b$ for each point $b$ of $B$. A subset $B$ of the topological space $X$ is \em locally collared \em
if it can be covered by a collection of open subsets $B_i$ of $B$ such that
each $B_i$ is collared. The main result of Brown \cite{Br} is
\begin{theorem}[\cite{Br}]\label{thm:Brown}
A locally collared subset of a metric space is collared.
\end{theorem}
Reading the proof of this theorem, we observe that this result 
can be achieved definably if we start with definable data. 

\medskip
Brown Theorem provides a direct proof of our Theorem \ref{conical-structure}. Indeed,
\begin{proof}[Proof of Theorem \ref{conical-structure}]
If $(X,0)$ is blow-spherical equivalent to its tangent cone, then its tangent link consists only of
simple points. Conversely

\smallskip
We equip the spherical modification  $\tilde{X}$ of $(X,0) $and its boundary $L(X,0)$ with the Euclidean metric.

Brown Theorem implies the existence of a positive real number $\nu$ and of a definable
homeomorphism $\Phi: L(X,0)\times [0,\nu[ \to \mathcal{V}$ where $\mathcal{V}$ is a definable open neighborhood of
$L(X,0)$ in $\tilde{X}$. Moreover we also know that $\Phi(\bu,0) = (\bu,0)$. That is enough to finish the proof.
\end{proof}

Let $X$ be a normally embedded definable set. Suppose that $(X,0)$ has a fast loop \cite{BF3,Fe1} or a choking horn \cite{BFGoS}, 
then there is a direction $v\in L_0X$ who is not a simple direction. Whenever the germ $(X,0)$ has a local separating 
set $Y\subset X$ at $0$ and $v\in L_0X$, the direction $v$ must belong to the singular locus of the tangent cone of $Y$ at $0$ 
(see \cite{BFN1,BFN2}) and thus it cannot be a simple direction.

\bigskip
\noindent{\bf Notation.} Let $L_G(X,0)$ denotes the set of all simple directions of $X$ at $0$ and $L_F(X,0)=L_0X\setminus L_G(X,0)$.

\begin{proposition} ${\rm dim} L_F(X,0)< {\rm dim}_0X -1$
\end{proposition}
 
\begin{definition} {\rm Let $c(V)\subset\R^n$ be the straight cone over $V\subset S(0,1)$ with vertex $0$. 
The $(a,\alpha)$-horn neighborhood of $c(V)$ is the following set}
$$ W_{(a,\alpha)}(c(V))=\{ x\in \R^n \ |  d(x,c(V) ) \leq a \|x\|^{\alpha}\}.
$$
\end{definition}

\begin{theorem}[Thin-Thick Decomposition]\label{thm:TtNE}
Let $X$ be a normally embedded definable set with an isolated singularity at $0\in X$. 
There exists a number $\beta >1$ such that the germ of $X$ at is decomposed in the following form
$$X=X_F\cup X_G$$ 
where $X_F=X\cap W_{(a,\beta)}(c(L_F))$ and $X_G=X\setminus X_F$. Moreover, $X_G$ is homeomorphic to a cone over $L_G(X,0)$.
\end{theorem} 

\begin{proof} Consider a family of pairs $(M_{R,r}^F,M_{R,r}^G)$ defined by:
$$M_{R,r}^F=\{ x\in X\cap S(0,R) \ | \ d(x,L_F)\leq r \} \quad \mbox{and} \quad M_{R,r}^G= X\cap S(0,R)\setminus M_{R,r}^F.$$ 
By the Theorem of Hardt \cite{Har,Co}, there exists a definable partition of $\R^2_+$ (parameter space) such that for each element of the 
partition, the family is topologically trivial along this element. In our 2-dimensional situation, the elements of this 
partition are 1-dimensional definable curves or 2-dimensional subsets, bounded by these curves. 
Let $\Omega$ be an element of this partition, bounded by an arc, tangent to $R$-axes and by a curve, who is not tangent 
to $R$-axes. Any arc $\sigma=(R,\sigma(R))$ belonging to $\Omega$ with the endpoint $0$ defines a pair of subsets of $X$:
$$X_{\sigma}^F=\{x\in X\cap S(0,R) \ | \ d(x,c(L_F))\leq \sigma(R)\} \quad \mbox{and} \quad X_{\sigma}^G=X\setminus X_{\sigma}^F.$$

\noindent{\bf Claim 1.} {\it For all $\sigma\in\Omega$, the sets $(X_{\sigma}^F,X_{\sigma}^G)$ are homeomorphic.}

\begin{proof}[Proof of the Claim 1] The sets are definable sets and their links are $(M_{R,\sigma(R)}^F,M_{R,\sigma(R)}^G)$. 
By the Local Topological Conical Structure of definable sets, we conclude the proof.
\end{proof} 

Let $\beta >1$ be a rational number such that the arc $(R,bR^{\beta})$ belongs to $\Omega$ for some $b>0$. 
Let $X_F$ and $X_G$ be defined by the arc $(R,bR^{\beta})$.

\medskip

\begin{proposition} $X_G$ is homeomorphic to the cone over $L_G$.
\end{proposition}
We will need the following 
\begin{lemma} Let $X$ be a definable set. Let $Y\subset X$ be a compact definable subset defined as the zero set of a definable function $f:X\rightarrow\R$. Then, for sufficiently small $\epsilon>0$, the sets $X\setminus Y$ and $\{x\in X \ | \ f(x)>\epsilon \}$ are homeomorphic.
\end{lemma}

\begin{proof}[Proof of the proposition]  Consider the set $X\setminus W(a,1)(c(L_F))$, for sufficiently small $a>0$. 
The link of this set at origin is $M_{R,aR}^G$. We showed that this set is homeomorphic to the link of $X_G$. By the 
lemma this set has the link homeomorphic to $L(X,0)\setminus L_F$. That is why $X_G$ is homeomorphic to the cone over $L_G$.
\end{proof}
The proof of Theorem \ref{thm:TtNE} is now complete.
\end{proof}

\begin{theorem}\label{thm:main-NE-XY}
Let $(X,0)$ and $(Y,0)$ be germs of normally embedded definable sets.

\smallskip
1)  If $(X,0)$ and $(Y,0)$ 
are blow-spherical equivalent, then $X_G$ and $Y_G$ are homeomorphic and $X_F\setminus 0$ has 
the same homotopy type of $Y_F\setminus 0$.

\smallskip
2) If the dimension of $X$ is lower than or equal to $5$, then the thin zones $X_F\setminus 0$ and 
$Y_F \setminus 0$ are also homeomorphic.
\end{theorem}

\begin{proof}First, if a mapping $\psi: (X,0)\rightarrow (Y,0)$ is a blow-spherical homeomorphism, 
then for any sufficiently small conical neighborhood $W_{(b,1)}(L_F(Y))\cap Y$ there exists a small conical neighborhood 
$W_{(a,1)}(L_F(X))\cap X$ such that $\psi(W_{(a,1)}(L_F(X))\cap X)\subset W_{(b,1)}(L_F(Y))\cap Y$. Indeed, 
if this claim were not true there would exist a sufficiently small conical neighborhood $W_{(b,1)}(L_F(Y))\cap Y$ 
such that $\psi(W_{(a,1)}(L_F(X))\cap X)\not \subset W_{(b,1)}(L_F(Y))\cap Y$. In other words, the Curve Selection Lemma
would guarantee the existence of a continuous definable arc $\gamma$ on $(X,0)$, tangent to $L_F(X)$, such that its image 
by $\psi$ is outside the  conical neighborhood $W_{(b,1)}(L_F(Y))\cap Y$. 
We would get a contradiction since the tangent direction of $\psi(\gamma)$ is simple 
while the tangent direction of $\gamma$ is not.

Let $M_F(X)$ be the boundary of the link of $X_F$ and $M_F(Y)$ be the boundary of the link of $Y_F$. The set of 
points of $X\cap S(0,\epsilon)$ situated between two small cones $W_{(a,1)}(L_F(X))\cap X$ and $W_{(\tilde{a},1)}(L_F(X))\cap X$ 
is homeomorphic to a cylinder over $M_F(X)$. By the lemma, $M_F(Y)$ divides the cylinder into two connected components 
such that  the different boundaries belong to different connected components. By the remark above, the same is true 
for the cylinder over $M_F(X).$ That is why the link of $X_F$ and the link of $Y_F$ have the same homotopy type. 

\smallskip
When ${\rm dim}_0X\leq 5$, the h-cobordant manifolds $M_F(X)$ and $M_F(Y)$ are homeomorphic.
\end{proof}

\begin{remark} {\rm The Thin-Thick Decomposition from \cite{BNP} is obtained from our Thin-Thick Decomposition in the following way. 
Consider a normal complex surface singularity and consider its real normal embedding. The set $L_F$ is thus obtained by 
the lifting of the essential exceptional directions, since all non-exceptional directions and non-essential exceptional 
directions correspond to simple points.}
\end{remark}

\section{General Case}

Let $J_{k,s}$ be the collection of half-planes of $\R^{k+2}$ defined as 
the Cartesian product of $\R^k$ with the following $s$ half-lines of $\R^2$ 
$$\bigcup_{\lambda=1}^s\{(x,y) \quad | \quad x\geq 0 \quad \mbox{and} \quad y=\lambda x\}$$
This object looks like a pamphlet with $s$ pages (of dimension $k$).

\medskip
Let $(X,0)$ be the germ of a definable set in $\R^n$, not necessarily normally embedded and with an isolated singularity. 
Consider the spherical modification $\tilde{X}$ of $X$. 
Let $L_0X$ be the tangent link of $(X,0)$. 
A point $v\in L_0X$ is called a \emph{simple direction} if the germ of $\tilde{X}$ at $v$ is homeomorphic to the germ $(J_{s,k},0)$ 
for some $s\in\N$ and where $k+1=\rm{dim}_0X$. The integer $s$ is called the \emph{directional multiplicity} of $(X,0)$ at $v$. 
Let $L_G(X,0)$ be the set of simple directions of $(X,0)$ and let $L_G^s(X,0)$ be the set of simple directions of multiplicity 
$s$. The family  $\{ L_G^s(X,0) \}_s$ defines clearly a partition of  $L_G(X,0)$.

\begin{remark}
{\rm The set of simple directions $L_G(X,0)$ is a definable open subset of $L_0X$.}
\end{remark}
The next result although obvious from the definitions is worth mentioning.
\begin{proposition} Let $(X,0)$ and $(Y,0)$ be germs of definable sets. Let $\Psi : \tilde{X}\rightarrow\tilde{Y}$ be 
a blow-spherical homeomorphism. For any interger $s$ we find $\Psi(L_G^s(X,0))=L_G^s(Y,0)$ for any integer $s$. 
In particular, $\Psi(L_G(X,0))=L_G(Y,0)$.
\end{proposition}

Let $X$ be a definable subset of $\R^n$ (not necessarily normaly embedded). 
According to \cite{BM}, there exist a subset $X'\subset\R^m$ and a map $\Phi :X'\rightarrow X$ such that

\begin{enumerate}
\item $X'$ is normally embeddd in $\R^m$; The subset $X'$ is called a \emph{normal embedding} 
of $X$.
\item The map $\Phi$ is definable and bi-Lipschitz with respect to the inner metric.
\end{enumerate}
By the results of the previous section, $X'$ has a thin-thick decomposition. The normal embedding mapping 
$\Phi$ induces a mapping $\tilde{\Phi}: \tilde{X'}\rightarrow \tilde{X}$ as follows. 
Let $v\in L_0X'$ and let $\gamma$ be a definable arc on $X'$ such that $v$ is the unit tangent vector of $\gamma$ at $0$. We define 
$\tilde{\Phi}(v)$ as the unit tangent vector of $\Phi(\gamma)$ at $0$. This construction is analogous to the construction 
presented by Bernig and Lytchak \cite{BL}. It is also clear that $\tilde{\Phi}$ is a definable map. 

Picking up a pancake decomposition of $X$, one can show that $\tilde{\Phi}$ is a Lipschitz finite map. Indeed, this
map is bi-Lipschitz on each pancake. The results of \cite{BL} and \cite{BFN2} imply that $\tilde{\Phi}$ is bi-Lipschitz 
on each pancake. Since, a pancake decomposition is finite, the map $\tilde{\Phi}$ is also finite.

Given $v\in L_0X$, the number $mult_v(X,0)$ is the \emph{multiplicity of $X$ at $0$ along $v$}, 
defined as the number of pre-images of $v$ by $\tilde{\Phi}$. 

Consider a direction $v\in L_0X$ as a point of the spherical modification of $(X,0)$. 
Let $N_v(X,0)=(\tilde{X}\cap B(v,\epsilon)\setminus L_0 X)$ the intersection of $\tilde{X}$ with 
the Euclidean ball $B(v,\epsilon)$ of radius $\epsilon$ and centered at $v$ .

\begin{proposition} 
The multiplicity  $mult_v(X,0)$ is the number of connected components of $N_v(X,0)$.
\end{proposition}

\begin{proof} Take $v\in L_0X$. The pre-images of $v$ are obtained in the following way: consider a maximal 
collection of arcs $\gamma_1,\dots,\gamma_s$ on $(X',0)$ such that, (for $i\neq j$), $\gamma_i$ is not tangent 
to $\gamma_j$, but their images by $\Phi$ are tangent arcs in $(X,0)$ with the unit tangent vector $v$.
Each connected component of the set $(\tilde{X}\cap B(v,\epsilon)\setminus L_0 X)$ contains an arc of 
the family $\{ \tilde{\Phi}(\gamma_i) \}$ and only one since $\tilde{\Phi}$ is finite.
\end{proof}

In particular, we proved that the inverse image of a simple point contains only simple points. On the other 
hand, the image of a simple direction $v'\in L_0X'$ by a normal embedding map may be not a simple point.
 (see example).


Note that $L_G(X,0)=\bigcup_{s=0}^S L_G^s(X,0)$ where $S$ is the maximal value of the multiplicities of simple points. 

\begin{proposition} For each $s$, ${\rm dim} L_G^s(X,0)={\rm dim}X-1$ if $L_G^s$ is not empty. 

Let $L_F(X,0)=L_0X\setminus L_G(X,0)$. 
Then, ${\rm dim}L_F(X,0)<{\dim}X-1$.
\end{proposition}

Let $L_G^s(X')=\tilde{\Phi}^{-1}(L_G^s(X,0))$. Remember that $\Phi$ is a normal embedding map and $\tilde{\Phi}$ is 
the extension of $\Phi$ to $\tilde{X'}$.

\begin{theorem}[Thin-Thick Decomposition] There exists $\beta>1$ such that the decomposition $X_F=X\cap W_{(a,\beta)}(c(L_F(X,0)))$ and $X_G=X-X_F$ satisfies the following:
$$X_G=\bigcup_{s=0}^S X_G^s$$
Each $X_G^s$ has tangent cone of dimension ${\rm dim}X-1$ and each $X_G^s$ is homeomorphic to a cone over $L_G^s(X')$.
\end{theorem}

\begin{proof} Proceeding as we did in the proof of the Thin-Thick Decomposition in normally embedded case, 
one can find a horn like neighborhood $W_{(a,\beta)}(c(L_F(X)))$ of $L_F(X)$ such that 
$X\setminus W_{(a,\beta)(c(L_F(X)))}$ has the same topology as $X\setminus W_{(\bar{a},1)(c(L_F(X)))}$. 
Let $X''$ be the pre-image of $X\setminus W_{(\bar{a},1)(c(L_F(X)))}$ by the map $\Phi$. 
This set can be presented as a union of the sets $(X')_G^s$ defined as the 
union of the subsets of $X''$ having the tangent cone belonging to $L_G^s(X')$.

By the Thin-Thick Decomposition Theorem for normally embedded sets, all the subsets $(X')_G^s$ are homeomorphic 
to their tangent cones. Thus, $(X')_G^s$ is homeomorphic to $L_G^s(X')$. Since the map $\Phi$ is a homeomorphism, 
$X_G^S$ is also homeomorphic to $L_G^s(X')$. This proves the theorem.
\end{proof}

\begin{theorem}\label{thm:main-nNE-XY}
Let $(X,0)$ and $(Y,0)$ be two definable  blow-spherical equivalent germs. Then
\begin{enumerate}
\item $X_G$ is homeomorphic to $Y_G$;
\item Each $X_G^s$ is homeomorphic to $Y_G^s$;
\item The sets $X_F\setminus 0$ and $Y_F\setminus 0$ have the same homotopy type.
\end{enumerate}
\end{theorem}
\begin{proof} This proof is the same as the proof of the Thin-Thick Decomposition Theorem for normally embedding sets.
\end{proof}

\end{document}